\documentclass[12pt,reqno,twoside]{amsart}
\usepackage{amsmath,amssymb,amsthm}				
\usepackage{verbatim}								
\usepackage{amscd}								
\usepackage[usenames]{color}		
\usepackage{hyperref}	
\newtheorem{Thm}{Theorem}[section]
\newtheorem{Cor}[Thm]{Corollary}						
\newtheorem{Prop}[Thm]{Proposition}					
\newtheorem{Lemma}[Thm]{Lemma}

\theoremstyle{definition}

\numberwithin{equation}{section}								
\def\da{Diophantine approximation}
\def\q{\mathbb{Q}}								
\def\r{\mathbb{R}}								
\def\z{\mathbb{Z}}								
					
\def\x{\mathbf{x}}
\def\n{\mathbb{N}}								

\def\h{\mathbb{H}}	
\def\ve{\varepsilon}

\newcommand\inv{^{-1}}								
\newcommand {\ignore}[1]  {}
\newcommand\lb{\left\langle}								
\newcommand\rb{\right\rangle}			
\newcommand\no{\noindent}

\def\SL{\operatorname{SL}}
\def\SO{\operatorname{SO}}
\def\BA{\operatorname{BA}}

\def\rpq{r_\alpha ({\bf p},q)}
	
\def\ggm{G/\Gamma}

\newcommand\comm[1]{\textcolor{red}{#1}}

\def\e{{\bf e}}
\def\u{{\bf u}}
\def\v{{\bf v}}
\def\x{{\bf x}}

\def\pq{\frac{{\bf p}}{q}}							
\def\nz{\smallsetminus\{0\}}

\newcommand\hs{homogeneous space}

\def\hd{Hausdorff dimension}

\newcommand\eq[2]{\begin{equation}\label{eq:#1}{#2}\end{equation}}
\newcommand {\equ}[1]     {\eqref{eq:#1}}

\title{Rational Approximation on Spheres}
\author{Dmitry Kleinbock \\ Keith Merrill}
\date{May 22, 2013}
\address{Department of Mathematics, Brandeis University, Waltham MA 02454, USA}
\email{kleinboc@brandeis.edu, merrill2@brandeis.edu}

\begin{document}

\maketitle

\begin{abstract} 
\no We quantify the density of rational points in the unit sphere $S^n$, proving analogues of the classical theorems 
on the embedding of $\q^n$ into $\r^n$. Specifically, we prove a Dirichlet theorem stating that every point $\alpha \in S^n$ is sufficiently approximable, the optimality of this approximation via the existence of badly approximable points, and a Khintchine theorem showing that the Lebesgue measure of approximable points is either zero or 
full
depending on the convergence or divergence of a certain sum. These results complement and improve on previous results, particularly recent theorems of Ghosh, Gorodnik and Nevo.
\end{abstract}

\section{Introduction}
\subsection{Motivation}
The field of Diophantine approximation seeks to {\it quantify} the density of a subset $A$ in a metric space $X$. Classical examples include the density of $\q$ in $\r$ or of a number field $K$ in its completion. 
One can also study the density of rational points in certain subsets $X$ of $\r^{m}$, specifically level sets of 
rational quadratic forms on $\r^{m}$. In this paper we analyze the case of spheres $S^n$, deferring the general case (of quadric hypersurfaces in $\r^{n+1}$) to a forthcoming work \cite{FKMS}, see \S\ref{alt methods} for more detail. Rational points on the sphere can be represented as $\pq$ with $q \in \n$ and ${\bf p} \in \z^{n+1}$ primitive. We want to measure the distance between $\alpha \in S^n$ and such a point $\pq$ against its complexity $q$. Unless otherwise specified, we will use the supremum norm $\|\cdot\|$ on $\r^{n+1}$ to measure distance.

It will be convenient to introduce the following general definition: for a  
subset $X$ of $\r^{m}$ and a function $\phi:\n \to (0,\infty)$, say that 
 $\alpha \in X$ is {\sl $\phi$-approximable in $X$\/}  if 
there exist 
infinitely many $({\bf p},q)\in\z^{m+1}$ with $\pq \in X$ such that \eq{phia}{\left\| \alpha - \pq \right\| <  \phi(q)\,;} 
the set of  points $\phi$-approximable in $X$ will be denoted by  $A(\phi,X)$.
Note that rational points  are $\phi$-approximable in $X$ for any positive function $\phi$, and if $\alpha$ is irrational then `infinitely many $({\bf p},q)\in\z^{m+1}$ with $\pq \in X$' can be replaced by `infinitely many $\pq \in \q^{m}\cap X$'. 

The requirement  $\pq \in X$ distinguishes the above set-up from the one usually considered in \da\ on manifolds -- there one studies rates of approximation of points on a manifold $X$ by rational points which do not have to lie in $X$. In other words, in this paper we are studying {\it intrinsic\/} \da\ on manifolds, as opposed to the existing very rich theory of 
approximation by rational points of the ambient space, see e.g.\ \cite{B, BD, KM3}.

The classical case $X = \r^m$ can be considered as a motivation. With the  notation \eq{phi tau}{\phi_\tau (x) := x^{-\tau}\,,} 
we have the following 
basic facts, see \cite{Schmidt-1980}:

\begin{itemize}
\item[$\bullet$] Dirichlet's Theorem on simultaneous \da\ 
implies that any $\x\in\r^m$ is 
$\phi_{1 + 1/m}$-approximable. 
\item[$\bullet$] For sufficiently small $c > 0$ the complement of $A(c\phi_{1 + 1/m}, \r^m)$ is non-empty; in fact, the union of complements to $A(c\phi_{1 + 1/m}, \r^m)$, called the set of {\sl badly approximable\/} vectors, has full \hd.

\item[$\bullet$] Khintchine's Theorem asserts that, when $x\mapsto x\phi(x)$ is non-increasing\footnote{In fact, the theorem holds for $\phi$ decreasing \cite{BDV}.}, almost every (resp.\ almost no) $\x\in\r^m$ is $\phi$-approximable 
assuming the sum $$\sum_{k= 1}^\infty \big(k\phi(k)\big)^m$$ diverges (resp.\ converges). 
\item[$\bullet$] By a theorem of Jarn\'ik, the set of $\phi_\tau$-approximable points, 
where $\tau$ is at least $ 1 + 1/m$,  has \hd\ $\frac{m+1}\tau$.
\end{itemize}


In this paper we will study the case $X = S^n$, the Euclidean unit sphere
in $\r^{n+1}$. The question of intrinsic approximation on spheres has been studied in the literature
 implicitly by Dickinson-Dodson  \cite{DD} and Drutu \cite{Drutu}\footnote{Both \cite{DD} and \cite{Drutu} study approximations by rationals in $\r^{n+1}$, but use the algebraic nature of $S^n$ to reduce their problems to 
intrinsic approximation.}, and explicitly by Schmutz \cite{Schmutz} and Ghosh-Gorodnik-Nevo \cite{GGN1, GGN2} (we note that in the two latter papers the generality is much wider, the subject being 
$S$-rational points on homogeneous varieties).  
We also mention that rational approximations of points on $S^2$ and $S^3$ can be obtained from the construction of Lubotzky-Phillips-Sarnak \cite{LPS, LPS2} of dense subgroups of $\SO(3)$ with entries in $\z[\frac1{p}]$ having optimal spectral gap. For $n \geq 4$, Oh \cite{Oh}, following earlier work of Clozel \cite{Clozel},  constructs subsets of $\SO(n)$ with lower bounds on their spectral gap, extending the previous construction. These quantitative equidistribution statements give rise to quantitative density statements of Hecke points, and thus a fortiori rational points. 
The method of the present paper is different: we use a connection between \da\ on spheres and dynamics/geometry of the 
quotient of $G = \SO(n+1,1)$ by a lattice $\Gamma$
and deduce intrinsic analogues of a number of basic results in \da, strengthening what has been known before. We note that all the results of this paper can also be derived by an alternative approach: relating intrinsic approximation on $S^n$ to the set-up of approximation of limit points of a lattice in $\SO(n+1,1)$ by its parabolic fixed points, and using results from  \cite{BDV, HV, Pat, SV} and a recent preprint \cite{FSY}. This approach was used in \cite{Drutu} and will be further elaborated upon in \cite{FKMS}.

\subsection{Statement of results}

Our main theorems are as follows. The first one gives an analogue of Dirichlet's theorem (see Theorem \ref{true dirichlet} for a stronger statement):

\begin{Thm} \label{special dirichlet}
There exists a constant $C \geq 1$ such that for every $\alpha \in S^n$, there exist infinitely many rationals $\pq \in S^n$ such that $$\left\| \alpha - \pq \right\| < \frac{C}{q}.$$
\end{Thm}

Previously Fukshansky \cite{Fukshansky} used a theorem of Hlawka \cite{Hlawka} about approximations of real numbers by Pythagorean triples to establish Theorem \ref{special dirichlet} in the special case of $S^1$, and showed that one can take $C = 2\sqrt{2}$. In \cite{Schmutz} a version of the above theorem was established for all $n$ with $\phi_1$ replaced by $\phi_{1/2\lceil{\log_2(n+1)}\rceil}$ and with an explicit dependence of $C$ on $n$. Later Ghosh, Gorodnik and Nevo \cite{GGN1} did the same with $\phi_1$ replaced by $\phi_{\frac14-\ve}$ for all $n$ and with $C$ dependent on $\alpha$ and $\ve$ (see \S\ref{dt} for a more precise statement of their results).

We show that $C$ in the above theorem cannot be replaced by an arbitrary small constant, by considering the set of $\alpha \in S^n $ which are {\it badly approximable in }$S^n$, that is, not $c\phi_{1}$-approximable in $S^n$ for some $c > 0$:
$$
\BA(S^n) := \left\{ \alpha \in S^n : \exists c = c(\alpha) \text{ such that } \forall \pq, \left\| \alpha - \pq \right\| > \frac{c}{q} \right\}.
$$
Analogously to Dani's result \cite{Da1} on the correspondence between simultaneous \da\ 
and homogeneous actions, we show that  
 $\alpha\in\BA(S^n) $ if and only if a certain trajectory on 
 $\ggm$
 is bounded. Then, using 
 \cite{Da2}, we establish

\begin{Thm} \label{ba} The set $\BA(S^n)$ 
is  thick. 
\end{Thm}
 
Here and hereafter we say that a subset of a metric space $X$ if {\it thick\/} if its intersection with any nonempty open subset of $X$ has full \hd.  

\medskip
A correspondence with dynamics also helps us to derive an analogue of Khintchine's Theorem, from which, in particular, it follows that $\BA(S^n)$ has Lebesgue measure zero.
Indeed, note that $
A(\phi,X)$ 
is  the limsup set of the family of balls $$\left\{B\left(\pq, \phi(q)\right): \pq \in S^n\cap \q^{n+1}\right\}.$$ Since up to a constant the Lebesgue measure of $B\big(\pq, \phi(q)\big)$  is $\phi(q)^n$, it is a consequence of the Borel-Cantelli Lemma that if the sum $
\sum_{\pq \in S^n} \phi(q)^n $
 converges, then the Lebesgue measure of $A(\phi,S^n)$ is zero. Furthermore, it follows from \cite{HB} that 
 \eq{count}{\# \left\{ \pq \in S^n \cap \q^{n+1} : q \leq N\right\} \ll  N^n} for  all $N > 0$ (here and hereafter $\ll$ means that the left hand side is bounded from above by the right hand side times a constant possibly dependent on $n$). 
 We refer the reader to \cite{Browning} for a nice introduction on counting rational points on varieties, and to \cite{Drutu} where counting results are derived from equidistribution of translates of horocycles. Given the above estimate,  one can deduce the following convergence-type statement for a non-increasing function $\phi$: 
\begin{eqnarray*}
\displaystyle\sum_{\pq \in S^n \cap 
\q^{n+1}} \phi(q)^n &=& \displaystyle\sum_\ell \displaystyle\sum_{\pq \in S^n \cap 
\q^{n+1},\,q \in (2^\ell, 2^{\ell+1}]} \phi(q)^n \\ &\leq& \displaystyle\sum_\ell \# \left\{ \pq \in S^n \cap 
\q^{n+1} : q \in (2^\ell, 2^{\ell+1}]\right\} \phi(2^\ell)^n\\ &\ll& \displaystyle\sum_\ell  (2^{\ell+1})^n \phi(2^\ell)^n \ll \int^\infty_{\ell=1} (2^\ell)^n \phi(2^\ell)^n d\ell \\ & \ll& \int^\infty_{\ell=1}  (2^\ell)^{n-1} \phi(2^\ell)^n d(2^\ell)  \ll \displaystyle\sum_k k^{n-1} \phi(k)^n.
\end{eqnarray*}
Therefore, if the series
\eq{sum}{\sum_{k=1}^\infty k^{n-1} \phi(k)^n} converges, it follows that almost no point $\alpha \in S^n$ is $\phi$-approximable. 

\medskip
The following theorem furnishes the converse result:
\begin{Thm} \label{khintchine} For any 
$\phi: \n \to (0,\infty)$ such that 
\eq{extraregularity}{\text{the function }k\mapsto k\phi(k)\text{ is non-increasing,}}
the Lebesgue measure of $A(\phi,S^n)$ is full (resp.\ zero) if and only if the sum \equ{sum} diverges (resp.\ converges).
\end{Thm}


We point out that Ghosh, Gorodnik, and Nevo \cite{GGN2} have recently proven various Khintchine-type results for intrinsic approximation on 
homogeneous varieties. 
In particular, they show that if
that if for some $a > c \cdot n$ (where $c$ is an explicitly computable constant $\geq 2$), $$\displaystyle\sum_{\pq \in S^n \cap 
\q^{n+1}} \phi(q)^a = \infty\,,$$ then the Lebesgue measure of $A(\phi,S^n)$ is full. Although, as noted previously, the results of \cite{GGN2} are more general, for approximations by rational points our result is much stronger, providing an exact converse to the convergence case above.

\medskip
 
As was suggested to us by V.\ Beresnevich, using the notion of mass transference developed in \cite{BeresnevichVelani_mass_transference} it is possible to strengthen Theorem \ref{khintchine} to obtain the following Hausdorff measure version:

\begin{Thm}\label{h measures sphere}
Let  $\phi$ be as in Theorem \ref{khintchine} and let $f \colon (0,\infty) \to (0,\infty)$ be a dimension function (see \S\ref{jt} for a definition) such that 
\eq{extraregularityhd}{\text{the function }k\mapsto k^nf\big(\phi(k)\big)\text{ is non-increasing,}}
Then the $f$-dimensional Hausdorff measure of $A(\phi,S^n)$ is full (resp.\ zero) if and only if the sum \eq{sumhd}{\sum_{k=1}^\infty k^{n-1} f\big(\phi(k)\big)} diverges (resp.\ converges). Consequently, for any $\tau \ge 1$ and $\phi_\tau$ defined by \equ{phi tau}, the \hd\ of $A(\phi_\tau,S^n)$ is equal to $n/\tau$. 
\end{Thm}


This reproves the results of \cite{Drutu}  in the case of unit spheres (see \cite[Theorems 1.1 and 4.5.7]{Drutu}, for $n=1$  it was done previously by 
 Dickinson and Dodson \cite{DD}). Note that with an alternative approach\footnote{suggested to us by S.\ Velani} based on the notion of ubiquity, the assumptions \equ{extraregularity} and \equ{extraregularityhd}  of Theorems \ref{khintchine} and \ref{h measures sphere} respectively can be weakened to just the monotonicity of $\phi$; see \S\ref{alt methods}  and a forthcoming work \cite{FKMS} for more detail.

\medskip


The correspondence that is instrumental in deriving all the aforementioned results is not new; it was already implicitly used 
in \cite{Drutu}.
However, to the best of our knowledge, it has never been stated explicitly before.  
 We now describe this correspondence and introduce the main ideas behind our proofs. 
 Let $Q:\r^{n+2} \to \r$ be the quadratic form 
 given by \eq{q}{Q({\bf x}) = 
 \sum^{n+1}_{i=1} x_i^2- x_{n+2}^2\,.} Then one can embed $S^n$ into the lightcone 
 $$L:=\{{\bf x}\in \r^{n+2} : Q({\bf x})  = 0\}$$ 
 of $Q$ via $\alpha \mapsto ({\alpha},1)$. Under this embedding, each rational point $\pq \in S^n$ determines a line in $L$ and a unique primitive vector $({\bf p},q) \in \z^{n+1} \times \n$ lying on this line. By Lemma \ref{close vector} below, good approximants $\pq$ to $\alpha \in S^n$ correspond under this mapping to lattice points $({\bf p}, q) \in \z^{n+2} \cap L$ which are close to the line through $({\bf \alpha},1)$. 
 Note that we have changed our approximating points from a dense subset to a discrete one, which dynamics is better equipped to handle. Let $\Lambda_0 := \z^{n+2} \cap L$. 

Denote by $G$ the group $\SO(Q)$ of orientation-preserving linear transformations which preserve $Q$.
Let $r_\alpha \in G$ denote an element which preserves 
$\r^{n+1}\times\{0\}$
and sends $({\bf \alpha},1) $ to $(1,0,\dots,0,1) \in L$ -- such an element is not unique for $n > 1$,
see \S\ref{corr} for more details on the choice of $r_\alpha$.
Applying $r_\alpha$ to the lightcone $L$, we see that good approximants $({\bf p},q) \in \Lambda_0$ become points in $r_\alpha \Lambda_0$ which are close to the line through $(1,0,...,0,1)$. Let $g_t \in G$ be a flow which contracts this line exponentially and expands the line through $(-1,0,...,0,1)$ exponentially (see $\S 2.1$ for an explicit description of $g_t$). Then 
points in $r_\alpha \Lambda_0$ close to this line correspond to small vectors in the lattice $g_t r_\alpha \Lambda_0$ for some $t \geq 0$. This is the central idea of this chapter, and the precise {\it quantitative\/} nature of this correspondence is the subject of Lemmas \ref{small vector} and \ref{close vector} below. 

By a {\it lattice\/} in $L$  we will mean a set of the form $g\Lambda_0$ for some $g \in G$. Let $\Gamma$ denote the stabilizer of $\Lambda_0$ in $G$. Then $\Gamma$ is a lattice in $G$, containing the subgroup $\SO(Q)_\z$ of integer points of $G$ as a finite index subgroup.
The space of lattices in $L$ can be identified with 
$\mathcal{L} :=\ggm$, a \hs\ with finite $G$-invariant (Haar) measure.  
Also let us define a function $\omega$ on 
$\mathcal{L}$ by $$\omega(\Lambda) := \displaystyle\min_{0 \ne {\bf v} \in \Lambda} \left\|{\bf v} \right\|.$$

The correspondence between approximation and dynamics is described in the following theorem, 
which is a partial analogue of Theorem $8.5$ in \cite{KM}. 

\begin{Thm} \label{dictionary} Let $\phi: [x_0,\infty) \to (0,\infty)$ be a piecewise $C^1$ function 
satisfying \equ{extraregularity}. 
Put  \eq{t_0}{t_0 = \ln\left(\frac{2}{\sqrt{n+1}\phi(x_0)}\right)\,,} and define a function  $\rho:[t_0,\infty) \to (0,\infty)$ by \eq{r(t)}{
\rho(t) = e^{-t} \cdot \phi\inv\left(\frac{2}{\sqrt{n+1}e^t}\right).
}
Then 
$\rho$ is non-increasing and the following hold:
\begin{itemize}
\item[$\bullet$] If $\alpha \in A(\phi,S^n)$, then there exists a sequence $t_k \to \infty$ such that $\omega(g_{t_k} r_\alpha \Lambda_0) < 2\rho(t_k)$;
\item[$\bullet$] If there exists a sequence $t_k \to \infty$ such that $\omega(g_{t_k}r_\alpha \Lambda_0) < \rho(t_k)$, then $\alpha \in A(\sqrt{n+1}\phi,S^n)$.
\end{itemize}
\end{Thm}
 In other words, up to constant, $\alpha$ is $\phi$-approximable if and only if the orbit $g_t r_\alpha \Lambda_0$ hits the `shrinking target' parametrized by $\rho(t)$ infinitely often.

\subsection{Outline of the Paper}
In  \S\ref{corr} we analyze the quantitative nature of the correspondence between good approximants $\pq$ to $\alpha$ and lattices $g_t r_\alpha \Lambda_0$ with small vectors. This analysis culminates in the proof of Theorem \ref{dictionary} which allows us to change our perspective from approximations on $S^n$ to 
properties of trajectories on $\mathcal{L}$.  
In \S\ref{rt} we study the geometry of the space $\mathcal{L}$ by means of reduction theory, and prove a version of Mahler's compactness criterion (Corollary \ref{mahlers}), thus establishing that small values of the function $\omega$ correspond to complements of large compact subsets of $\mathcal{L}$. 

Then in \S\ref{proofs} we prove our main results. 
In \S\ref{dt} we combine the correspondence of \S\ref{corr} with Mahler's criterion  to prove Theorem  \ref{special dirichlet}, now reduced to a statement about lattices in $L$. In fact we  prove a stronger statement, Theorem \ref{true dirichlet}, which establishes the so called  uniform $(C,\frac{1}{2},\frac{1}{2})$-Dirichlet property of every  $\alpha \in S^n$. 
In \S\ref{bdd} we  
derive 
Theorem \ref{ba} from Dani's result on bounded geodesics on finite volume hyperbolic manifolds. 
In  \S\ref{kt} we recall the framework set forth by Kleinbock and Margulis in \cite{KM} to establish a Borel-Cantelli lemma about cuspidal penetrations. We conclude that the set of lattices whose trajectories penetrate a sequence of shrinking cuspidal neighborhoods infinitely often is either null or full depending on the convergence or divergence of the sum of measures of these neighborhoods. We then estimate these measures and, using the correspondence defined in Theorem \ref{dictionary}, relate their sum to the convergence or divergence of \equ{sum}. After that in \S\ref{jt} we recall the machinery of mass transference developed by \cite{BeresnevichVelani_mass_transference} to deduce Theorem \ref{h measures sphere} from Theorem \ref{khintchine}. 

Lastly in \S\ref{alt methods}, we discuss other techniques which can be used to  prove our theorems, and mention some generalizations which will appear in a forthcoming paper \cite{FKMS}.

\subsection{Acknowledgements}
The authors are grateful to Cornelia Drutu, Lior Fishman, Alex Gorodnik, Hee Oh, David Simmons, Sanju Velani, and Victor Beresnevich for many helpful discussions. The work of the first named author was supported in part by NSF grant DMS-1101320. 


\section{Good approximations and small vectors}\label{corr}

Let $\{{\bf u}_i \}$ denote the standard basis on $\r^{n+2}$  with respect to which $Q$ has the familiar form \equ{q}.
We will refer to the group of orientation-preserving linear transformations preserving $Q$ as $\SO(Q)$, and denote it by $G$. 

Let ${\bf e}_1$ denote the vector ${\bf u}_1 + \u_{n+2}$ and let $K\cong \SO(n+1)$ be the subgroup of $G$ preserving $ \operatorname{Span}(\u_1,\dots,\u_{n+1})$.
For $\alpha \in S^n \subset \operatorname{Span}(\u_1,\dots,\u_{n+1})$, we would like  to choose an element $r_\alpha \in K$ such that $r_\alpha \alpha = {\bf u}_1$, or, equivalently, $r_\alpha(\alpha,1) = {\bf e}_1$. As mentioned previously, for $n > 1$ such an element is not unique. However, if we map $K$ to $S^n$ via $g \mapsto g({\bf u_1})$, then the stabilizer of $\u_1$ in $K$ is isomorphic to $\SO(n)$, identified with the lower right $n \times n$ block of $\SO(n+1)$. Therefore 
there is a unique coset $r_\alpha\inv \SO(n)$ with the property that $g{\bf u}_1 = \alpha$ for any $g\in r_\alpha\inv \SO(n)$. Our first goal is to choose a particular section  $$S^n\cong\SO(n+1)/\SO(n) \to K\,.$$ 
Note that without loss of generality we can restrict our attention to an open neighborhood $W$ of the hemisphere of $S^n$ centered at ${\bf u}_1$, since the union of $W$ and its image under reflection covers $S^n$, and all the Diophantine properties we consider are invariant under reflection.

\ignore{
Now let us recall the structure of the Lie algebra $\mathfrak{g} = \mathfrak{so}(n+1,1)$. 
The Lie algebra $\mathfrak{k} := \mathfrak{so}(n+1)$ of $K$ consists of skew-symmetric matrices. Then $$\mathfrak{g} = \left\{ \begin{pmatrix} {\bf b} & {\bf x^T} \\ {\bf x} & 0 \end{pmatrix} : {\bf b} \in \mathfrak{k},\  {\bf x} \in \r^{n+1} \right\}.$$ Consider the $n$-dimensional subspace $\mathfrak{s}$ of $ \mathfrak{k}$ corresponding to
$${{\bf x } = {\bf 0}, \qquad {\bf b} = \begin{pmatrix} 0&{\bf y}\\{\bf -y^T} & 0 \end{pmatrix}, \ {\bf y} \in \r^n\,.}$$
that is, 
\eq{def s}{\mathfrak{s} = \left\{ {\bf s}_{\bf y} :=\begin{pmatrix} 0 & {\bf y} & 0 \\ {\bf -y^T} & 0 & 0\\ 0 & 0 & 0 \end{pmatrix}: {\bf y} \in \r^n \right\}.}

Then it is easy to see that for each $\alpha \in W$ there exists a unique element $g \in K$ such that $\ln(g) \in \mathfrak{s}$ and $g({\bf e}_1) = (\alpha,1)$.  We define $r_\alpha := g\inv$, so that  $r_\alpha(\alpha,1) = \e_1$. Furthermore, it is clear, and will be used later, that the map $\r^n\to W$ given by 
\eq{map}{{\bf y} \mapsto\big( \exp({\bf s}_{\bf y})\big)\inv({\bf e}_1)}
is bi-Lipschitz on the preimage of $W$ (here the latter is viewed as a subset of the hyperplane $\{x_{n+2} = 1\}$). }
\medskip

Let \eq{g_t}{g_t := \begin{pmatrix}  \cosh(t) & 0 & -\sinh(t)\\ 0 & I_n &0\\ -\sinh(t) & 0 & \cosh(t)\end{pmatrix} \in G\,,} 
and let  $$A := \{ g_t : t \in \r \}\,.$$
Then one easily checks: $$g_t \e_1 = \begin{pmatrix}  \cosh(t) & 0 & -\sinh(t)\\ 0 & I_n &0\\ -\sinh(t) & 0 & \cosh(t)\end{pmatrix} \begin{pmatrix} 1\\0\\ \vdots\\ 0 \\ 1 \end{pmatrix} = e^{-t} \e_1.$$
Let us also define the {\it horospherical subgroups\/} associated to $\{g_t\}$. These subgroups capture the dynamically significant behavior of the $g_t$-action. Namely:
\begin{itemize}
\item the {\it contracting} subgroup 	$U := \{ h \in G : g_t h g_{-t} \to e \text{ as } t \to \infty\}$;
\item the {\it neutral} subgroup $H^0 := \{ h \in G : g_t h = h g_t  \text{ for all } t\}$;
\item the {\it expanding} subgroup	$H := \{ h \in G: g_{-t} h g_t \to e \text{ as } t \to \infty\}$.
\end{itemize}

One knows that $G$ is locally a product of $U$, $H^0$ and $H$ (that is, the Lie algebra of $G$ is the direct sum of the Lie algebras of these three subgroups). Additionally, we recall the Iwasawa decomposition of $G$:
\begin{Thm} \label{iwasawa}
The mapping $U \times A \times K \to G$ is a diffeomorphism. 
\end{Thm}

The next lemma constructs a section $W\to K$ mentioned above:

\begin{Lemma}\label{biLip} There exist two bi-Lipschitz maps $W \to K$ and $W \to H$ which we will denote by $\alpha \mapsto r_\alpha$ and $\alpha \mapsto h_\alpha$, where $W \subset S^n$ is a neighborhood of the hemisphere containing ${\bf u}_1$, such that for any $\alpha\in S^n$ one has
\eq{ralpha}{r_\alpha(\alpha,1) = {\bf e}_1 } and
\eq{agrees}{
h_\alpha r_\alpha\inv  \in U 
H^0\,.}
\end{Lemma}

\begin{proof} 
To prove the lemma we first need to better understand the structure of $H$, the  subgroup of $G$ whose Lie algebra is given by 
\eq{lieh}{\mathfrak{h} 
:= \left\{ \begin{pmatrix} 0 & -{\bf  x}^T & 0 \\ {\bf x} & 0 & {\bf x}\\ 0 & {\bf  x}^T & 0 \end{pmatrix}: {\bf x} \in \r^n \right\}\,.}
By Theorem \ref{iwasawa}, every element $h \in H$ can be uniquely represented as \eq{defsigma}{h = u g_s k\,,} where $u\in U$, $s\in \r$, $k\in K$. Let $\sigma: H \to K$ be the projection onto $K$, i.e. $\sigma(h) = k$, where $h$ and $k$ are as in  \equ{defsigma}. This mapping is injective: if we have two elements $h = u g_s k$ and $h' = u' g_t k$ for which $k = \sigma(h) = \sigma(h')$, then $$H \ni h' \cdot h\inv = u' g_{t} k \cdot k\inv g_{-s} u\inv =  u''g_{t-s} \in U A\subset UH^0 \,.$$ Since $H \cap UH^0 $ is trivial, we have $s = t$ and $u = u'$. Clearly $\sigma$ is locally bi-Lipschitz.

One readily checks that $U\bf e_1 = \bf e_1$ (indeed, a change of coordinates identifies $U$ as a subgroup of upper triangular matrices), and, as mentioned previously, $g_s{\bf e}_1 = e^{-s} {\bf e}_1$. Therefore, with $\sigma(h) = k$ one has $$h\inv{\bf e}_1 = k\inv g_{s}\inv u\inv {\bf e}_1 = k\inv g\inv_{s} {\bf e}_1 = e^{s} k\inv{\bf e}_1.$$ Since an element $h \in H$ is uniquely determined by the image of ${\bf e}_1$, it follows that the mapping $$H \to H{\bf e}_1, \qquad h \mapsto h\inv{\bf e}_1$$ is locally bi-Lipschitz. In this way we can view $H$ as an $n$-dimensional submanifold of the light cone $L$. 
Explicitly, using \equ{lieh} one can parametrize this embedding  as 
$$H{\bf e}_1 = \left\{ {\bf v}_{{\bf x}} := \begin{pmatrix} 1 - \left\| {\bf x} \right\|^2 \\ 2x_1\\ \vdots \\ 2x_n \\ 1 + \left\| {\bf x} \right\|^2 \end{pmatrix} \right\}.$$ Since we have an embedded copy of $S^n \subset L$ given by points whose last coordinate is $1$, we obtain a map $\pi: H{\bf e}_1 \to S^n$ given by linearly scaling ${\bf v}_{{\bf x}}$ by $1/(1 + \left\| {\bf x} \right\|^2)$. This map is locally bi-Lipschitz. 

We can now define the desired maps. Given $\alpha \in W$, there exists a unique $\bf x \in \r^n$ such that $\pi({\bf v}_{{\bf x}}) = (\alpha,1) \in S^n$. Now define $h_\alpha \in H$  such that $h_\alpha\inv{\bf e}_1 = {\bf v}_{{\bf x}}$, and let $r_\alpha := \sigma(h_\alpha)$.

Since both maps are bi-Lipschitz on $W$, it remains to show \equ{ralpha} and
\equ{agrees}. Note that $$r_\alpha\inv{\bf e}_1 = S^n \cap \langle h_\alpha\inv{\bf e}_1 \rangle = \pi({\bf v}_{{\bf x}}) = (\alpha,1)$$ as needed; and \equ{agrees} follows because, in view of \equ{defsigma}, $h_\alpha r_\alpha\inv \in U A \subset U  H^0$.\end{proof}

As mentioned in the introduction, the key idea behind our proofs is to restate the problem of approximating $\alpha \in S^n$ as a problem  of approximating the line 
 through $\e_1$
by the lattice $r_\alpha \Lambda_0 \in \mathcal{L}$. Applying the flow $g_t$ contracts $\e_1$, and by continuity, good approximants to this line will correspond, for some time $t \geq 0$, to short vectors. We now quantify that relationship.


\medskip

Our results will be stated with respect to the sup norm, but due to the definition of $L$, it is often more convenient to work with the Euclidean norm on $\r^{n+2}$, denoted by $\left\| \cdot \right\|_e$. By the equivalence of norms on $\r^m$, this changes the estimate only by a universal constant. Explicitly, $$\left\| {\bf x} \right\| \leq \left\| \x \right\|_e \leq \sqrt{m} \left\| \x \right\| \text{ for any } \x \in \r^m.$$ It is worth noting, however, that for points on $L$ we have a better approximation. If $\x \in L$, then by definition $$Q({\bf x}) = 0 \Leftrightarrow x^2_{n+2} = 
\displaystyle\sum^{n+1}_{i=1} x^2_i.$$ From this it immediately follows that $$\left\| \x \right\| = |x_{n+2}| \text{ and } \left\| \x \right\|_e = \sqrt{2} |x_{n+2}|.$$ We will use these estimates frequently in what follows. 

Accordingly, since we are interested in the norms of vectors in $L$ under $g_t$, we compute \eq{g_t height}{|(g_t \x)_{n+2}|= \frac{1}{2}\left| e^t(x_{n+2}-x_1) + e^{-t} (x_{n+2} + x_1)\right|.}

Finally, 
we will have frequent need for an estimate on the term $\left\| q\cdot {\bf e_1} - \rpq \right\|_e$. By definition, $$\left\| q\cdot \e_1 - \rpq \right\|_e^2 = (q - \rpq_1)^2 + \displaystyle\sum^{n+1}_{i=2} \rpq^2_i.$$ Since $\rpq \in L$, we have that $q^2 = \sum^{n+1}_{i=1} \rpq^2_i$.  Combining these estimates shows that $$\left\| q \cdot \e_1 - \rpq \right\|_e = \sqrt{2q(q-\rpq_1)}.$$ 

\medskip

We can now justify our remarks that good approximants correspond to small vectors. 

\begin{Lemma} \label{small vector} Let $N \geq q$ be such that $\left\| \alpha - \pq \right\| < \frac{\ve}{\sqrt{qN}} \leq \frac{\ve}{q}$. Then there exists $t > 0$ such that $\left\| g_{t} r_\alpha ({\bf p},q) \right\|  < \ve \sqrt{n+1} \sqrt{\frac{q}{N}} \leq \ve \sqrt{n+1}$.
\end{Lemma}
\begin{proof} By our computations above, we have that because $\left\| \alpha - \pq \right\| < \frac{\ve}{\sqrt{qN}}$, the same is true of the Euclidean norm up to a factor, namely $$\left\| \alpha - \pq \right\|_e < \sqrt{n+1} \frac{\ve}{\sqrt{qN}}\,.$$ 
Multiplying both sides by $q$ and noting that $r_\alpha$ is a Euclidean isometry, we have $$\left\| q \e_1 - \rpq \right\|_e < \ve\sqrt{n+1} \sqrt{\frac{q}{N}}\,.$$ Now observe that if $\alpha - \pq  = 0$, then $g_t \rpq = g_tq\e_1\to 0$ as $t\to\infty$, so the conclusion of the lemma holds trivially. Otherwise, let $t_*$ be the unique point in time when the 
distance from $g_t \rpq$ to the origin is minimized  --  explicitly, this occurs at \eq{tstar}{t_* = \frac{1}{2}\ln\left(\frac{q+\rpq_1}{q-\rpq_1}\right).} For $t = t_*$, we compute 
\begin{eqnarray*}
\left\| g_{t_*} \rpq \right\| &=& |\big(g_{t_*} \rpq\big)_{n+2}| = \sqrt{q^2 - \rpq_1^2}\\ &\leq& \sqrt{2q\big(q-\rpq_1\big)} =  \left\| q{\bf e_1} - r_\alpha({\bf p},q)  \right\|_e\\ &<&  \ve\sqrt{n+1} \sqrt{\frac{q}{N}} \leq  \ve\sqrt{n+1}. 
\end{eqnarray*} \end{proof}

\begin{Lemma} \label{close vector} If for some $t > 0$, $\left\| g_t r_\alpha ({\bf p},q) \right\| < \delta$, then there exists an $N > q$ such that $\left\| \alpha - \pq \right\| < \frac{2\delta}{\sqrt{qN}}$.
\end{Lemma}
\begin{proof} 
If we set $N = e^t \delta$,  then we must have $q < N$ (this comes from comparing the norms of $g_t \rpq$ and $g_t q \e_1$). By the chain of inequalities $$\left\| \alpha - \pq \right\| \leq \left\| \alpha - \pq \right\|_e  = \frac{1}{q} \left\| q \e_1 - \rpq \right\|_e = \frac{1}{q} \sqrt{2q\big(q - \rpq_1\big)}\,,$$ it suffices to estimate the term $q-\rpq_1$. But by \equ{g_t height}, $$\frac{1}{2}e^t\big(q-\rpq_1\big) \leq |g_t \rpq_{n+2}|< \delta,$$ from which it immediately follows that $$q-\rpq_1 < 2e^{-t}\delta = \frac{2\delta^2}{N}.$$ Plugging this estimate back into the above, we obtain $$\left\| \alpha - \pq \right\| < \frac{1}{q}\sqrt{2q\big(q-\rpq_1\big)} < \frac{1}{q} \sqrt{2q\frac{2\delta^2}{N}} = \frac{2\delta}{\sqrt{qN}},$$ as needed.
\end{proof}
 


Given the above results, we have that a specific approximant $\pq$ satisfying \equ{phia} corresponds to a time $t_*$ when $$\omega(g_{t_*} r_\alpha \Lambda_0) < \sqrt{n+1} q\phi(q)\,,$$ and conversely if $\left\| g_t \rpq\right\| < q\phi(q)$, then $\left\| \alpha - \pq \right\| < 2 \phi(q)$. Moreover, if $\alpha \notin 
\q^{n+1}$, then $r_\alpha \Lambda_0 \cap \lb \e_1 \rb = \varnothing$. The significance of this trivial observation is that, whenever $\alpha$ is irrational,  for every element $({\bf p},q) \in \Lambda_0$ one has $$\left\| g_t r_\alpha ({\bf p},q) \right\| \to \infty \text{ as } t \to \infty.$$ In particular if $\alpha \notin 
\q^{n+1}$ and $\phi$ is decreasing, then any given approximant $\pq$ only works for a bounded length of time. 

It therefore seems reasonable to try and define a non-increasing function $\rho(t)$ with the property that $$\rho(t) = q \cdot \phi(q),$$ where 
$t$ is such that $g_t r_\alpha ({\bf p},q)$ is closest to the origin.
Indeed, this almost works except that $t_*$ in \equ{tstar}
depends on all the coordinates of $({\bf p},q)$ as well as on $\alpha$, not just on $q$. Our goal now is to approximate $t_*$ by a value of $t$ depending only on $q$. By our previous estimates on the Euclidean norm, if $\pq$ and $\alpha$ satisfy \equ{phia}, then  \eq{estimate}{(n+1)q^2 \phi(q)^2  > 2q\big(q-\rpq_1\big)\,.} 

Define $t_q = \ln\left(\frac{2}{\sqrt{n+1} \phi(q)}\right)$, and then define $\rho(t)$ such that $\frac{2}{\sqrt{n+1}} \rho(t_q) = q   \phi(q)$. This gives rise precisely to the expression 
\equ{r(t)}.
Clearly, if $\phi(x)$ is defined on $[x_0,\infty)$, then $\rho(t)$ is defined on $[t_0,\infty)$, where $t_0$ is given by \equ{t_0}.

\medskip

 We can now prove Theorem \ref{dictionary}.
\begin{proof}[Proof of Theorem \ref{dictionary}]
Let us first address the case of $\alpha \in \q^{n+1}\cap S^n$, say $\alpha = \pq$. As mentioned before, $\pq$ is $\phi$-approximable in $S^n$ for any positive function $\phi$, so it remains to show that there is an unbounded sequence $t_k$ such that $\omega(g_{t_k} r_{{\bf p}/q}\Lambda_0) < \rho(t_k)$. In fact, we will show this estimate holds for all $t$ sufficiently large, with $\pq$ as its own approximant:
\begin{eqnarray*}
\omega(g_{t} r_{{\bf p}/q}\Lambda_0) &\leq& \left\| g_t r_{{\bf p}/q} ({\bf p},q) \right\|\\ &=& \left\| g_t (q,0,...,0,q) \right\|\\ &=& q \cdot e^{-t} \\ &<& e^{-t} \phi\inv\left(\frac{2}{\sqrt{n+1}e^t}\right)\\ &=& \rho(t),
\end{eqnarray*}
where these inequalities hold whenever $\phi\inv\left(\frac{2}{\sqrt{n+1}e^{t}}\right) > q$, i.e. $t > \ln\left(\frac{2}{\sqrt{n+1}\phi(q)}\right)$. 

\medskip
Now suppose that $\alpha \in S^n$ is irrational and $\phi$-approximable  in $S^n$, and let $\pq \in S^n$ satisfy \equ{phia}. We will show that $\left\| g_{t_q} r_\alpha ({\bf p},q) \right\| < 2 \rho(t_q)$: 
\begin{eqnarray*}\left\| g_{t_q} r_\alpha ({\bf p},q) \right\| &=& \frac{1}{2}\left(e^{t_q} (q -\rpq_1) + e^{-t_q} (q + \rpq_1)\right) \\ &\leq& \frac{1}{2}\left( \frac{2}{\sqrt{n+1}\phi(q)} \frac{(n+1) q \phi(q)^2}{2} + \frac{\sqrt{n+1}\phi(q)}{2} 2q\right)\\ &=& \sqrt{n+1} q \phi(q)\\ &=& 2 \rho(t_q).
\end{eqnarray*}
Note that our use of \equ{estimate} 
is legitimate since our assumption on $\pq$ and $\alpha$ implies that $\left\| q \e_1 - r_\alpha ({\bf p},q) \right\|_e $ is less than $\sqrt{n+1} q \phi(q)$. 

Conversely, suppose that the lattice $g_t r_\alpha \Lambda_0$ contains a vector of length less than $\rho(t)$, and let $g_t \rpq$ be such a vector. First note that we must have $$q \leq e^t \rho(t)$$ (this follows from comparing the norm of $g_t r_\alpha ({\bf p},q)$ with that of $g_t (e^t \rho(t) \cdot \e_1)$, and noting that the norm of $g_t (e^t \rho(t) \cdot \e_1)$ is precisely $\rho(t)$). Furthermore, by Lemma \ref{close vector}, we have that $$\left\| \alpha - \pq \right\| < \frac{2\rho(t)}{q}\,.$$ So it suffices to prove that $\rho(t) \leq q \phi(q)$. But $$q \leq e^t \rho(t) = \phi\inv\left( \frac{2}{\sqrt{n+1}e^{t}}\right).$$ Let $s = \phi\inv\left(  \frac{2}{\sqrt{n+1}e^{t}}\right)$. Since the function $x \mapsto x  \phi(x)$ is assumed to be non-increasing, we have 
\begin{eqnarray*}
q\cdot \phi(q) \geq& s \phi(s) =  \frac{2}{\sqrt{n+1}e^{t}} \phi\inv\left( \frac{2}{\sqrt{n+1}e^{t}}\right)= \frac{2}{\sqrt{n+1}} \rho(t),
\end{eqnarray*}so that $\left\| \alpha - \pq \right\| < \frac{2\rho(t)}{q} \leq \sqrt{n+1} \phi(q)$. 

In fact this last argument shows that the function $\rho(t)$ is non-increasing: let $t < t'$, we claim that $\rho(t) \geq \rho(t')$. Indeed, $\frac{2}{\sqrt{n+1}e^{t}} > \frac{2}{\sqrt{n+1}e^{t'}}$, and since $\phi$ is decreasing, $$s = \phi\inv\left(\frac{2}{\sqrt{n+1}e^{t}}\right) < \phi\inv\left(\frac{2}{\sqrt{n+1}e^{t'}}\right) = s'\,.$$ Since $x \mapsto x\phi(x)$ is non-increasing, we have $s \phi(s) \geq s'\phi(s')$, which immediately yields $\rho(t) \geq \rho(t')$ as needed.

We observed previously that for every $({\bf p},q) \in \Lambda_0$, $\left\| g_t r_\alpha ({\bf p},q) \right\| \to \infty$ as $t$ increases, and therefore each $\pq$ works only for a bounded length of time. Since the sequence $t_k$ is unbounded, there must be infinitely many distinct approximants, i.e. $\alpha \in A(\sqrt{n+1}\phi, S^n)$. 
\end{proof}

\section{Reduction theory}\label{rt}

For the proof of Theorem \ref{dictionary} we also need some background in reduction theory for $\ggm$.
We will use a rough fundamental domain for the action of $\Gamma$ on $G$ in terms of the Iwasawa decomposition (Theorem \ref{iwasawa}). For $\tau \in\r_+$, let 
${A_\tau := \{ g_s : s \geq -\ln(\tau) \}}$, and define  a {\it Siegel set\/} to be  a set of the form 
$${\mathfrak{S}_{\tau,M} = K A_\tau  M},$$ where $A_\tau$ is as above and $M \subset U$ is relatively compact. 

The following theorem shows that finitely many translates of some Siegel set give a rough fundamental domain for $\Gamma$:
\begin{Thm}[\cite{Bo}, \S 13; see also \cite{L}, Proposition 2.2]
\label{siegel}
There exists a Siegel set $\mathfrak{S} = \mathfrak{S}_{\tau,M}$ and a finite set $F = \{f_1,...,f_m\} \subset G \cap \SL_{n+2}(\q)$ such that the union $\Omega := \bigcup^m_{i=1} \mathfrak{S} f_i$ satisfies
\begin{enumerate}
\item $G = \Omega \Gamma$;
\item for any $f \in G \cap \SL_{n+2}(\q)$, the set $\{\gamma \in \Gamma : \Omega f \cap \Omega \gamma\ne\varnothing\}$ is finite.
\end{enumerate}
\end{Thm}

Our next goal is to relate the function $\omega$ 
to a metric on $\mathcal{L}$.
Actually, it will be more convenient to do it through the function $\Delta:\mathcal{L}\to\r$ given by  \eq{delta}{\Delta (\Lambda) := -\ln \omega(\Lambda)\,.}
 Choose a right-invariant and bi-$K$-invariant metric `$\operatorname{dist}_G$' on $G$, normalized so that $\operatorname{dist}_G(g_s,g_t) = |s-t|$. Also denote by `dist'   the induced metric on $\mathcal{L}= \ggm$, namely, define
$$\operatorname{dist}(g\Lambda_0, h\Lambda_0) :=\inf_{\gamma\in\Gamma}\operatorname{dist}_G(g,h\gamma)\,.
$$
Clearly one has $\operatorname{dist}(g\Lambda_0, h\Lambda_0) \le \operatorname{dist}_G(g,h)$.
A partial converse, where $g,h$ are taken from a Siegel set, is known as Siegel's conjecture, proved for $G = \SO(n+1,1)$ by Borel\footnote{Borel's proof is known to be incomplete for groups of higher rank, see \cite[Remark 5.6]{L}; however since $G$ is of real rank one, it is sufficient for our situation so we cite his result. See \cite[Theorem 7.6]{J} and \cite[Theorem 5.7]{L} for the correct proof of the general case. } 
\cite[Theorem C]{Bo}:

\begin{Thm}
\label{sigelconj} Le $ \mathfrak{S}$ and $F$ be as in Theorem \ref{siegel}. Then there exists a constant $D > 0$ such that for each $f\in F$, any $g \in  \mathfrak{S}f$ and any $\gamma \in \Gamma$, 
$$
\operatorname{dist}_G(e,g) - D \le \operatorname{dist}(g\Lambda_0,\Lambda_0)\le \operatorname{dist}_G(e,g) \,.
$$
\end{Thm}

Now we can state the desired relationship between $\Delta$  and $\operatorname{dist}$:

\begin{Lemma}
\label{dl} $\sup_{g\in G}|\operatorname{dist}(g\Lambda_0,\Lambda_0) - \Delta(g\Lambda_0)|<\infty$.
\end{Lemma}
\begin{proof} We are going to relate both functions in 
the statement of the lemma to the $A_\tau$-term of the Siegel decomposition of $g$ given by Theorem \ref{siegel}. By the theorem we have a description \eq{g decomp}{g = k g_s u f_i \gamma \in \Omega \Gamma\,,}
where  $f_i \in G \cap \SL_{n+2}(\q)$, $k\in K$, $s\ge -\ln(\tau)$, $u\in M\subset U$ and $\gamma\in\Gamma$. We first show that $\omega(g\Lambda_0) \asymp e^{-s}$  (here and hereafter we use notation $A \asymp B$ if $A \ll B \ll A$).
Let $N$ be a common denominator of all the matrix coefficients of $f_1,\dots,f_m$ and $f_1\inv,\dots,f_m\inv$. Then $Nf_i\inv \in \operatorname{GL}_{n+2}(\z)$, therefore  ${\bf w}:=N\gamma\inv f_i\inv{\bf e}_1\in \z^{n+2}\cap L = \Lambda_0$. Since ${\bf e}_1$ is fixed by $U$ and contracted by $g_s$, $s > 0$, we have that 
\begin{eqnarray*}
\omega(g\Lambda_0)\le \left\| g{\bf w}\right\| &=& \left\| k g_s u f_i\gamma (N\gamma\inv f_i\inv{\bf e}_1)\right\| 
\\&=& N\left\| k g_s {\bf e}_1 \right\|\\ &=& N\left\| k (e^{-s} {\bf e}_1) \right\|\\ &\ll& 
e^{-s}.
\end{eqnarray*}

To prove the other bound,
note that the terms $g_s$ contract scalar multiples of ${\bf e}_1$ faster than any other vectors in $L$, hence
for every ${\bf v}\in\Lambda_0\nz$ 
one has 
\begin{eqnarray*} \left\| g{\bf v} \right\|  &=&  \left\|  k g_s u f_i \gamma{\bf v} \right\| \gg \left\|  g_s u f_i \gamma{\bf v} \right\| \\ &=& \frac1{N} \left\|  g_s u (Nf_i) \gamma{\bf v} \right\| \ge  \frac1{N} e^{-s}\left\|  u (Nf_i) \gamma{\bf v} \right\| \\ &\ge& \frac1{N} e^{-s}\frac1{ \left\| u\inv \right\|}\left\|  Nf_i \gamma{\bf v} \right\| \ge \frac{e^{-s}}{ N\left\| u\inv \right\|}\end{eqnarray*}
(the last inequality holds since $Nf_i \gamma{\bf v}\in\z^{n+2}\nz$). But $u$ belongs to a compact subset of $U$, hence $\left\| u\inv \right\|$ is uniformly bounded from above; thus $\omega(g\Lambda_0) \gg e^{-s}$, as desired. In other words,  $\sup_{g\in G}| \Delta(g\Lambda_0) - s|<\infty\,,$ where $g$ and $s$ are as in \equ{g decomp}. In view of  Theorem \ref{sigelconj}, to prove Lemma \ref{dl} it remains to show that 
$$\sup_{f\in F,\,k\in K,\,u\in M,\,s \ge -\ln(\tau)}| \operatorname{dist}_G(kg_suf,e) - s|<\infty\,.$$
But this is immediate from the  invariance properties of the metric, 
compactness of $K$, boundedness of $MF$ and the normalization of $\operatorname{dist}_G$. \end{proof}


A consequence of the above lemma is a compactness criterion for subsets of $\mathcal{L}$, similar to 
Mahler's Compactness Criterion for $\SL_n(\r)/\SL_n(\z)$ \cite{M}. For $\ve > 0$, consider
$${\mathcal{K}_\ve := \{ \Lambda \in \mathcal{L} : \omega(\Lambda) \ge \ve \} = \{ \Lambda \in \mathcal{L} : \Delta(\Lambda) \le \log(1/\ve) \}\,. }$$

\begin{Cor} \label{mahlers} 
A subset $E\subset\mathcal{L}$ is relatively compact if and only if  $E\subset\mathcal{K}_\ve$ for some positive $\ve$. 
\end{Cor}
\begin{proof} The `only if' direction is straightforward by the continuity of $\omega$; for the other direction, it suffices to show that each  $\mathcal{K}_\ve$ is bounded, which is immediate from Lemma \ref{dl}.\end{proof}

Clearly ${\mathcal K}_0 = \mathcal{L}$ and if $\delta > \ve$, then ${\mathcal K}_\delta \subset {\mathcal K}_\ve$; thus $\{{\mathcal K}_\ve  : \ve > 0\}$ gives a compact exhaustion of $\mathcal{L}$. This makes it possible to interpret the correspondence of Theorem \ref{dictionary} as a connection between  good approximations of $\alpha\in S^n$ by rational points of $S^n$  and excursions of trajectories $g_tr_\alpha\Lambda_0$ in $\mathcal{L}$ outside of large compact subsets. 

\medskip
We 
close the section with another useful corollary:
%

\begin{Cor} \label{mc}
There exists $C > 0$ such that $K_C = \varnothing$.
\end{Cor}
\begin{proof} 
If no such constant $C$ existed, then for every $k \in \n$ we could find $\Lambda_k \in \mathcal{L}$ such that $\omega(\Lambda_k) > k$. By Corollary \ref{mahlers}, the collection $\{\Lambda_k : k \geq 1 \}$ is precompact, and hence has a limit point $ \Lambda$. Let $\v \in \Lambda$ be nonzero. By the topology on $\mathcal{L}$, there exist vectors $\v_k \in \Lambda_k$ such that $\v_k \to \v$. But this contradicts the fact that $\left\| \v_k \right\| > k$. 
\end{proof}

\section{Proofs of Theorems \ref{special dirichlet}, \ref{ba}, \ref{khintchine}, and \ref{h measures sphere}}\label{proofs}
\subsection{Dirichlet's Theorem}\label{dt}

Our goal for this subsection is to derive Theorem \ref{special dirichlet} from a stronger statement, Theorem  \ref{true dirichlet}. 
Let us introduce the following definition;
for a subset $X$ of $\r^{n+1}$ and   real numbers
 $C,a,b$ let us say that $\alpha \in X$ is 
\begin{itemize}
\item[$\bullet$] {\sl $(C,a,b)$-uniformly Dirichlet in $X$\/}  if given any $N > 1$  
\eq{dir}{\exists\,\pq\in X\text{ with }q \leq N\text{  such that }\left\| \alpha - \pq \right\| <\frac{C}{q^aN^b}.} 
\item[$\bullet$] {\sl $(C,a,b)$-Dirichlet in $X$\/}  if $\exists\,N_0$ such that \equ{dir} holds for $N > N_0$. 
\end{itemize}


In \cite{Schmutz} it was shown that every $\alpha \in S^n$ is $(C,0,b)$-uniformly Dirichlet in $S^n$ with 
$$C = 4\sqrt{2}\lceil{\log_2(n+1)}\rceil\quad\text{and}\quad b = \frac1{2\lceil{\log_2(n+1)}\rceil}\,.$$ Later 
a systematic study of this property for homogeneous varieties $X$ was undertaken in  \cite{GGN1}, where in particular it has been shown that
\begin{itemize}
\item[$\bullet$] every $\alpha\in S^n$ is $(1,0,b)$-Dirichlet in $S^n$ for any
$$\begin{cases}b < 1/4 & n   \text{ even}\\ b < 1/3 & n  = 3\\ b < \frac14 + \frac3{4n} & n \ge 5 \text{ odd} \end{cases}$$
 \item[$\bullet$] almost every $\alpha\in S^n$ is $(1,0,b)$-Dirichlet  in $S^n$, where
$$\begin{cases}b < 1/2 & n   \text{ even}\\ b < 2/3 & n  = 3\\ b < \frac12 + \frac3{2n} & n \ge 5 \text{ odd} \end{cases}$$
\end{itemize}

In this section we prove 

\begin{Thm} \label{true dirichlet} There exists a constant $C$ such that   $\forall\,\alpha \in S^n$ is $(C,1/2,1/2)$-uniformly Dirichlet in $S^n$. 
\end{Thm}

This, in particular,  implies being $(1,0,b)$-Dirichlet for any $b < 1/2$ and improves on all the aforementioned results  valid for every $\alpha$ (although for odd $n$,  \cite{GGN1}'s almost everywhere statements yield a still better approximation). Also it is clear that, for $b > 0$,  $(C,a,b)$-Dirichlet  implies $C\phi_{a+b}$-approximable; thus Theorem \ref{special dirichlet} immediately follows from Theorem \ref{true dirichlet}. Note that our value for $C$, coming from Corollary \ref{mc}, is not effective -- it would be interesting to get an explicit estimate. As mentioned previously, for $n=1$ it follows from \cite{Fukshansky} that we may take $C = 2\sqrt{2}$. 


\begin{proof}  
Let $C$ be the minimal constant making Corollary \ref{mc} true; clearly $C \geq 1$ as witnessed by the standard lattice $\Lambda_0$. Fix $\alpha \in S^n$ and let $N > C \geq 1$. We need to find $\pq \in S^n$ such that $$q \leq N \text{ and } \left\| \alpha - \pq \right\| < \frac{2C}{\sqrt{qN}}.$$ Let $t = \ln\left(\frac{N}{C}\right) > 0$, and consider the lattice $g_t r_\alpha \Lambda_0 \in \mathcal{L}$. By 
Corollary \ref{mc}, we have that $\omega(g_t r_\alpha \Lambda_0) \leq C$. Let $({\bf p},q) \in \Lambda_0$ be such 
that $\|g_t r_\alpha({\bf p},q)\|\le C$. Then 
$q \leq e^t  C = N$, and by Lemma \ref{close vector}, we have that $$\left\| \alpha - \pq \right\| \leq \frac{2C}{\sqrt{qN}}$$ as needed.

It remains to prove the inequality when $1 < N \leq C$. Let $\pq = \frac{1}{1}(1,0,...,0,1)$. Because the diameter of the sphere is $2$, for any $\alpha \in S^n$ we have   $$\left\|\alpha - \frac{1}{1}(1,0,...,0,1) \right\| \leq 2 < \frac{2C}{\sqrt{1 N}}$$ as desired.
\end{proof}

\subsection{$\BA(S^n)$ and the optimality of Theorem \ref{special dirichlet}}\label{bdd}

We now show that the function $\phi_1(q) = \frac{1}{q}$ appearing in Theorem \ref{special dirichlet}  is optimal in the sense that for any faster decaying function $\psi$, there are points in $S^n$ which are not 
$\psi$-approximable. Specifically, any badly approximable point will fail to be $\psi$-approximable. So to demonstrate the optimality of $\phi_1$, it suffices to show that $\BA(S^n)$ is nonempty. Indeed, we will show more, namely that this set is  thick.
The key ingredient here is a dynamical interpretation of the set $\BA(S^n)$. It will be convenient to define 
\eq{def b}{\mathcal{B}:= \big\{g \in G : \{g_t g \Lambda_0 : t \ge 0\}  \text{ is bounded in $\mathcal{L}$}\big\}\,.}

\begin{Prop}\label{ba is bdd}  $\alpha \in \BA(S^n)$  if and only if $r_\alpha\in\mathcal{B}$.
\end{Prop}
\begin{proof} First suppose that $\alpha \in \BA(S^n)$, i.e.\ there exists $\varepsilon > 0$ such that $\alpha$ is not in $A(\varepsilon \phi_1, S^n)$. Applying Theorem \ref{dictionary} with the function $\phi:= \frac{\varepsilon}{\sqrt{n+1}}\phi_1$, we have that for all $t$ sufficiently large, $$\omega(g_t r_\alpha \Lambda_0) > \rho(t),$$ where $\rho(t)$ is given by \equ{r(t)} (note in this case $\phi$ is its own inverse): $$\rho(t) = e^{-t}\phi\inv\left(\frac{2}{\sqrt{n+1}}e^{-t}\right) = e^{-t} \frac{\varepsilon}{\sqrt{n+1}} \frac{\sqrt{n+1}e^t}{2} = \frac{\varepsilon}{2},$$
independent of $t$. But this, by Corollary \ref{mahlers}, says precisely that the orbit $\{g_t r_\alpha \Lambda_0 : t \geq 0 \}$ is bounded in $\mathcal{L}$. 

Conversely, suppose that $r_\alpha\in\mathcal{B}$.
By Corollary \ref{mahlers} this is equivalent to the existence of $c > 0$ such that $\omega(g_t r_\alpha \Lambda_0) > c$ for every $t \geq 0$. Let $\phi := \frac{c}{\sqrt{n+1}}\phi_1$, then similarly to the above computaion, $\rho(t) = c/2$, therefore $$\omega(g_t r_\alpha \Lambda_0) > 2\rho(t) = c \text{ for all }t \geq 0\,.$$ By Theorem \ref{dictionary}, $\alpha$ is not contained in $ A\left(\frac{c}{\sqrt{n+1}} \phi_1, S^n \right)$, i.e. $\alpha \in \BA(S^n)$. 
\end{proof}

Now recall  the following theorem of  Dani  \cite{Da2}:


\begin{Thm}\label{dani} 
The set $H \cap \mathcal{B}$ is thick in $H$. 
\end{Thm}


Theorem \ref{ba} will follow from showing the set  $W  \cap\BA(S^n)$ to be bi-Lipschitz to a neighborhood in $H \cap \mathcal{B}$. 

%
\begin{proof}[Proof of Theorem \ref{ba}]



Let $W' \subset H$ be the image of $W$ under the mapping $\alpha\mapsto h_\alpha$
discussed in Lemma \ref{biLip}. 
Since \hd\ is preserved by bi-Lipschitz mappings, it remains to show that $W \cap \BA(S^n)$ is mapped bijectively to $W' \cap \mathcal{B}$. 

By Theorem \ref{ba is bdd}, we know that $\alpha \in \BA(S^n)$ if and only if $r_\alpha \in \mathcal{B}$, i.e. $$\{ g_t r_\alpha \Lambda_0 : t \geq 0 \} \text{ is bounded}.$$ But 
$$g_t h_\alpha \Lambda_0 = g_t h_\alpha r_\alpha\inv r_\alpha \Lambda_0 = (g_t h_\alpha r_\alpha\inv g_t\inv ) g_tr_\alpha \Lambda_0$$ is at a uniformly bounded distance from $g_tr_\alpha \Lambda_0$, since, by 
Lemma \ref{biLip}, $h_\alpha r_\alpha\inv $ is an element of $ U H^0$, the product of the neutral and contracting horospherical subgroups corresponding to $\{g_t : t \ge 0\}$. 
Thus $r_\alpha \in \mathcal{B}$ if and only if $h_\alpha \in \mathcal{B}$, i.e. $$h: W \cap \BA(S^n) \to W' \cap \mathcal{B}$$ is a bijection, as needed.
\end{proof}
Note that Dani proves this by establishing a stronger property: winning in the sense of Schmidt \cite{Schmidt games}. This has been recently strengthened by McMullen to so-called absolute winning, see \cite{McM} for details. Both winning and absolute winning properties are preserved by bi-Lipschitz mappings.
Consequently, Theorem \ref{ba} can be strengthened to an assertion that the set $\BA(S^n)$ is absolutely winning.

\subsection{Khintchine's Theorem}\label{kt}

We next prove the divergence case of Theorem \ref{khintchine}. Recall that we are given $\phi: \n \to (0,\infty)$ such that the function $k\mapsto k\phi(k)$  is non-increasing and the series 
\equ{sum} diverges. Since $\phi$ is decreasing, we may extend its domain from $\n$ to $[1,\infty)$ such that it is piecewise $C^1$ and the function $x\mapsto x\phi(x)$  is still non-increasing. In view of Theorem \ref{dictionary}, and replacing $\phi$ by $\frac1{\sqrt{n+1}}\phi$, to prove Theorem \ref{khintchine} it suffices to show that 
\eq{conclusion}{\text{for a.e. }\alpha\in S^n\ \exists \text{ a sequence }t_k \to \infty\text{ such that }\omega(g_{t_k}r_\alpha \Lambda_0) < \rho(t_k)\,,} where $\rho:(t_0,\infty) \to(0,\infty)$ is associated to $\phi(\cdot)$ as in  Theorem \ref{dictionary}. 
This will be a consequence of the following theorem -- 
a dynamical Borel-Cantelli Lemma describing the $g_t$-action on the space $(\mathcal{L},\mu)$, where $\mu$ stands for  the probability Haar measure on $\mathcal{L} $.

\begin{Thm} \label{dynamicalkhintchine} For any function 
$\rho: \n \to (0,\infty)$, 
\eq{set}{
\mu\big(\{\Lambda\in\mathcal{L} : \omega(g_{t}\Lambda) < \rho(t ) \text{ for infinitely many }t \in \n\}\big)= \begin{cases} 1 & \  \\ 0 &  \ \end{cases}
}
according to the divergence or convergence of the sum \eq{int}{\sum^\infty_{t=1} \rho(t)^n\,.} 
\end{Thm}

\begin{proof}[Proof of Theorem \ref{khintchine} assuming Theorem \ref{dynamicalkhintchine}]  
First let us have a lemma connecting \equ{sum} to \equ{int}:

\begin{Lemma}\label{phirho} Let $\phi(\cdot)$ and $\rho(\cdot)$ be related via \equ{r(t)}. Then \eq{simult}{\int^\infty_{t_0} \rho(t)^n \,dt < \infty \iff \int^\infty_{x_0} x^{n-1}\phi(x)^n \,dx < \infty\,.}
\end{Lemma}
\begin{proof} 
Using \equ{r(t)}, one can rewrite  \equ{int} as $$\left(\frac{\sqrt{n+1}}{2}\right)^n\int^\infty_{t_0} \left(\frac{2}{\sqrt{n+1}e^t}\right)^n \phi\inv \left(\frac{2}{\sqrt{n+1}e^t}\right)^n \,dt\,.$$ 
After a change of variable $ x = \phi\inv\left(\frac{2}{\sqrt{n+1}e^{t}}\right)$, the previous integral becomes equal to 
$$\int^\infty_{x_0}  \phi(x)^{n} x^n \left(-\frac{\phi'(x)}{\phi(x)}\right)\,dx = -\int^\infty_{x_0} x^n \phi(x)^{n-1} \phi'(x)\,dx\,,$$ 
which, after integration by parts, can be written as $$\int_{x_0}^\infty x^{n-1} \phi(x)^n dx + \frac{1}{n} x_0^n \phi(x_0)^n - \lim_{x \to \infty} \frac{1}{n} x^n \phi(x)^n.$$ But since the function $x\phi(x)$ is non-increasing,  the last term above is finite, and thus the two integrals in \equ{simult} converge or diverge simultaneously.
\end{proof}

Now back to the proof of Theorem \ref{khintchine}. As before, without loss of generality we can restrict our attention to $\alpha \in W$. Suppose that 
\equ{conclusion} fails, that is, there exists a subset $W_0$ of $W$ of positive measure consisting of $\alpha$ such that 
\eq{concl1}{\forall \,\alpha\in W_0,\quad 
\omega(g_{t}r_\alpha\Lambda_0) \ge \rho(t)\text{ for large enough }t\in\r\,.}
Now take a small neighborhood $B$ of identity in $UH^0$, recall the map $\alpha\mapsto h_\alpha$ from Lemma \ref{biLip}, 
and write, for $g\in B$,
$$
g_{t}gh_\alpha \Lambda_0 = g_{t}(gh_\alpha r_\alpha\inv)g_{-t}g_t  r_\alpha\Lambda_0\,.
$$
In view of \equ{agrees}, we have $gh_\alpha r_\alpha\inv$ is contained in $ UH^0$, and moreover, in a fixed (dependent on $B$ and $W_0$) subset of  $ UH^0$. Arguing as in the proof of Theorem  \ref{ba}, we see that there exists a compact subset $M$ of $G$ such that for any $\alpha\in W_0$ and $g\in B$,  one has $
g_{t}gh_\alpha \Lambda_0 = g' g_t  r_\alpha\Lambda_0$ for some $g'\in M$. This and \equ{concl1} imply the existence of a constant $c > 0$ such that 
$${\forall \,\alpha\in W_0\text{ and } g\in B,\quad
\omega(g_{t}gh_\alpha\Lambda_0) \ge c\rho(t)\text{ for large enough }t\in\r\,.}$$
But since the product map $U\times H^0\times H\to G$ is a local diffeomorphism and the map $\alpha\mapsto h_\alpha$ is  bi-Lipschitz, we can conclude, by Fubini's Theorem, that the Haar measure of $g\in G$ such that $\omega(g_{t}g\Lambda_0) \ge c\rho(t)\text{ for large enough }t$ is positive. Therefore the set in \equ{set}, with $\rho$ replaced by $c\rho$ and extended to $\n\cap[1,t_0]$ in an arbitrary way, does not have full measure. 
By Theorem \ref{dynamicalkhintchine}, the sum \equ{int} converges, and by the monotonicity of $\rho$, so does the integral $ \int^\infty_{t_0} \rho(t)^n \,dt $. Thus, by Lemma \ref{phirho} and the regularity of $\phi$, the sum \equ{sum} also converges, contradicting our assumption. \end{proof}


Now let us turn to  Theorem \ref{dynamicalkhintchine}.   Its convergence part (which we do not need for the proof of Theorem \ref{khintchine}) is a straightforward consequence of the Borel-Cantelli Lemma and the following fact:

\begin{Lemma}\label{measures} 
For all $\ve > 0$ one has
$$
\mu(\mathcal{L}\smallsetminus \mathcal{K}_\ve) = \mu \big(\{\Lambda \in\mathcal{L} : \omega(\Lambda ) < \ve\} \big)\asymp\ve^n\,.
$$ 
\end{Lemma}

As for  the divergence part, one needs to verify certain quasi-independence conditions on the $g_t$-preimages of sets $\mathcal{L}\smallsetminus \mathcal{K}_\ve$. Such methods date back to the work of Sullivan \cite{Sull} and Kleinbock-Margulis \cite{KM}; in fact we are going to derive 
Theorem \ref{dynamicalkhintchine} from one of the main results of  \cite{KM}:
  
  \begin{proof}[Proof of Theorem \ref{dynamicalkhintchine}]  
  From the continuity of the $G$-action on $L$ it follows that the function $\Delta $ defined in \equ{delta} is uniformly continuous, and Lemma \ref{measures} amounts to saying that 
  $$ 
  \mu(\{\Lambda \in\mathcal{L} : \Delta(\Lambda ) > z\}) \asymp  e^{-nz}
  \,.$$
  In other words, in the terminology of \cite{KM}, $\Delta$ is $n$-DL. Thus \cite[Theorem 1.7]{KM} applies, and one can conclude that the family of super-level sets of $\Delta$,
  $$ \big\{ \left\{ \Lambda  \in \mathcal{L}  : \Delta(\Lambda ) \geq z \right\} : z \in \r \big\}\,,
 $$ is {\it Borel-Cantelli\/} for $g_1$. The latter by definition means that for any sequence $\{E_t : t \in \n \}$ of sets from the above family one has $$ \mu \big(\{ \Lambda \in \mathcal{L}  : g_t(\Lambda ) \in E_t \text{ for infinitely many } t \in \n \} \big) = \begin{cases} 0 & \text{if }\sum^\infty_{t=1} \mu(E_t) < \infty\\ 1 & \text{ otherwise }\end{cases}$$
which is precisely the conclusion of Theorem \ref{dynamicalkhintchine} in view of Lemma \ref{measures}.
   \end{proof}

It remains to write down the
\begin{proof}[Proof of Lemma \ref{measures}]
In light of Theorem \ref{siegel} and  Lemma \ref{dl}, to prove Lemma \ref{measures} it suffices to show that 
for $\tau$ and $M$ as in Theorem \ref{siegel}, one has 
\eq{n-dl}{\mu(\{ g \in \mathfrak{S}_{\tau,M} : \operatorname{dist}_G(g,e) \geq z\}) \asymp e^{-nz}.} We remark that it follows from \cite[Lemma 5.6]{KM} that \equ{n-dl} holds with some explicitly computable $k$ in place of $n$. However, for completeness, we give the proof here. Consider the projection $G = K \times A \times U \to A$. The Haar measure on $G$ is pushed forward by this mapping to a measure proportional to $\delta(a) da$, where $da$ is the Lebesgue measure on $A$ and $\delta(a)$ is the modulus of conjugation by $a$ on $U$. Explicitly, we can compute $\delta(a)$ as follows: the Lie algebra $\mathfrak{u}$ of $U$ can be described in coordinates as $$\mathfrak{u} = \left\{ \begin{pmatrix} 0 & {\bf x}^T& 0\\ -{\bf x} & 0 & {\bf x}\\ 0 & {\bf x}^T & 0 \end{pmatrix} : {\bf x} \in \r^n \right\}.$$ Clearly $\operatorname{dim}(\mathfrak{u}) = n$. Conjugation by $g_s$ acts on $\mathfrak{u}$ as scalar multiplication by $e^{-s}$,  hence $\delta(g_s)$, which is the determinant of the map $\operatorname{Ad}(g_s): \mathfrak{u} \to \mathfrak{u}$, equals $e^{-ns}$. 
By the discussion of \cite[\S 5]{KM}, it suffices to show that $\operatorname{dist}_G$ on $A$ satisfies $$\int_{\{g_s: \operatorname{dist}_G(g_s,e) \geq z\}} e^{-ns} ds \asymp e^{-nz}.$$ But since $\operatorname{dist}_G(g_s,e) = s$, this immediately follows. 
\end{proof} 

\subsection{Mass Transference and Jarn\'ik's Theorem}\label{jt}

In this section we recall  the machinery of mass transference developed in \cite{BeresnevichVelani_mass_transference}, which allows one 
to derive Jarn\'ik-type results from Khintchine-type results. Specifically, we will derive Theorem \ref{h measures sphere} from Theorem \ref{khintchine}. 

To do this, we first need to establish some background.
A {\it dimension function\/} is a function $f \colon (0,\infty) \to (0,\infty)$ which is
increasing and continuous, and such that $\lim_{r\to0} f(r) = 0$. Given  a metric space $X$, a ball $B = B(x,r) \subset X$ and a dimension function $f$, we define its $f$-{\it volume} to be the quantity $$V^f(B) := f(r).$$ For a subset $A \subset X$ and $\ve > 0$, we define the quantity $$H^f_\ve(A) := \inf\left\{ \sum_i V^f(B_i) : A \subset \bigcup_i B_i, r(B_i) < \ve\right\}.$$ Denote by $H^f(A)$ the $f$-{\it dimensional Hausdorff measure} of $A$, given by $$H^f(A) = \lim_{\ve \to 0} H^f_\ve(A) = \sup_{\ve > 0} H^f_\ve(A).$$ 

\noindent Here and hereafter we let $r(B)$ denote the radius of the ball $B$. 


Given a ball $B = B(x,r) \subset X$ and a dimension function $f$, we will form a new ball $$B^f := B\big(x,f(r)^{1/n}\big).$$ 
The following theorem is a special case of \cite[Theorem 3]{BeresnevichVelani_mass_transference}:
\begin{Thm}[Mass Transference Principle]
\label{mass transference}
Let $\{B_i\}$ be a countable collection of balls in an $n$-dimensional manifold $X$  such that $r(B_i) \to 0$ as $i \to \infty$. Let $f$ be a dimension function such that $r^{-n} f(r)$ is monotonic, and suppose that for any ball $B \subset X$, \eq{lebesgue case}{H^n(B \cap \limsup_i B^f_i) = H^n(B).} Then for any ball $B \subset X$ one has \eq{h measure case}{H^f(B \cap \limsup_i B_i) = H^f(B).}
\end{Thm}

Note that $H^n$ is equivalent to Lebesgue measure, thus \equ{lebesgue case} is simply the statement that $\lambda(\limsup_i B^f_i) = 1$, where $\lambda$ denotes the normalized Lebesgue measure on the sphere. 

\begin{proof}[Proof of Theorem \ref{h measures sphere}]
The strategy  for deriving Theorem \ref{h measures sphere} from Theorem \ref{khintchine} is now clear: fix a dimension function $f$ and an approximating function $\phi$, and let the collection $\{B_i\}$ enumerate the balls $\left\{B\left(\pq, \phi(q)\right)\right\}$. Since there are only finitely many rationals with specified denominator, we may assume that $r(B_i) \to 0$ as $i \to \infty$. 

Since $B\left(\pq, \phi(q)\right)^f = B\left(\pq, f\big(\phi(q)\big)^{1/n}\right)$,  to verify \equ{lebesgue case} we need to apply Theorem \ref{khintchine} with $\phi$ replaced by $(f \circ \phi)^{1/n}$ (this is where the assumption \equ{extraregularityhd} is used). The Lebesgue measure of $A\big((f \circ \phi)^{1/n}, S^n\big)$ is either full or null depending on the divergence or convergence of the sum $$\sum_k k^{n-1} \left(f(\phi(k)^{1/n}\right)^n = \sum_kk^{n-1} f\big(\phi(k)\big)\,.$$ 

Suppose first that this sum diverges. Then \equ{lebesgue case} holds, and thus by Theorem \ref{mass transference}, we see that $$H^f\big(A(\phi,S^n)\big) = H^f(S^n),$$ as needed. 

Now suppose that the sum converges. We claim that this implies $$\displaystyle\sum_{\pq \in S^n} f\big(\phi(q)\big) < \infty\,.$$ Indeed, it follows from \cite{HB} that 
\begin{eqnarray*}
\sum_{\pq \in S^n} f(\phi(q)) &=& \sum_{\ell=1}^\infty \sum_{\{\pq \in S^n,\ q \in (2^{\ell-1}, 2^\ell]\}} f\big(\phi(q)\big)\\ &\ll& \sum^\infty_{\ell=1} 2^{\ell n} f\big(\phi(2^{\ell-1})\big)\\ &\ll& \int_{t\geq1} (2^{t-1})^n f(\phi(2^{t-1}))\, dt\\ &\ll& \int_{t\geq1} t^{n-1} f\big(\phi(t)\big) \,dt\\ &\ll& \sum_{k = 1}^\infty k^{n-1} f\big(\phi(k)\big) < \infty.
\end{eqnarray*}
Moreover, since $\phi(k) \to 0$ as $k \to \infty$, for any $\ve > 0$ we can choose $N$ such that $\phi(q)  < \ve$ for $q \geq N$. As a result, $$A(\phi,S^n) \subset \bigcup_{\{\pq \in S^n :\ q \geq N\}} B\left(\pq, \phi(q)\right)$$ is a cover which satisfies $r(B_i) < \ve$. It follows that $H^f_\ve\big(A(\phi,S^n)\big)$ is bounded above by the tail of a convergent sum. Hence $$H^f\big(A(\phi,S^n)\big) = \lim_{\ve \to 0} H^f_\ve\big(A(\phi,S^n)\big) = 0\,.$$ 

It remains to verify the last statement of the theorem. Since 
$$\sum_{r=1}^\infty r^{n-1} \phi_\tau(r)^{n/\tau} = \sum^\infty_{r=1} r^{n-1} (r^{-\tau})^{n/\tau} = \sum^\infty_{r=1} r^{-1} = \infty\,,$$
from the part we just proved it follows that $H^{n/\tau}\big(A(\phi_\tau,S^n)\big) = \infty$.  Similarly, for $s > \frac{n}{\tau}$ one has $$\sum^\infty_{k=1} k^{n-1} \big(\phi_\tau(k)\big)^s = \sum^\infty_{k=1} k^{n-1} k^{-s\tau} < \infty$$ because $-s\tau < -n$. Therefore, again by the main part of the theorem,  $H^s\big(A(\phi_\tau,S^n)\big) = 0$. Since $$\operatorname{dim}(X) = \sup\{d : H^d(X) = \infty\} = \inf \{d : H^d(X) = 0\},$$ we can conclude that $\operatorname{dim}\big(A(\phi_\tau,S^n)\big) = \frac{n}{\tau},$ finishing the proof.
 \end{proof}

\subsection{
Further developments}\label{alt methods}

In a follow-up paper \cite{FKMS}  currently in preparation, jointly with L.\ Fishman and D.\ Simmons we are addressing a more general problem of intrinsic approximation on arbitrary quadratic varieties. In particular, there we point out that  all the results of the present paper hold in the set-up of a rational quadratic variety $$X := \{\x\in\r^{n+1} : {\mathfrak q}(\x) = 1\}$$ such that 
$X\cap \q^{n+1}$ is non-empty and the associated quadratic form $$Q(x_1,\dots,x_{n+2}) := {\mathfrak q}(x_1,\dots,x_{n+1}) - x_{n+2}^2$$ has rank one. 
This can be shown via a connection between intrinsic approximation on $X$ as above and approximation of limit points of a lattice in $\SO(n+1,1)$ by its parabolic fixed points. Moreover, it follows that  rational points on those varieties form a {\it locally ubiquitous system\/}, see \cite{BDV} or \cite{Drutu} for a definition. The method of ubiquitous systems was used in \cite{Drutu} to derive an analog of Theorem \ref{h measures sphere} for general quadratic forms; as mentioned previously, this method allows one to replace  the assumptions \equ{extraregularity} and \equ{extraregularityhd} of Theorems \ref{khintchine}  and \ref{h measures sphere} with just the monotonicity of $\phi$.

\ignore{Namely, one can derive an analogue of Theorem \ref{true dirichlet}, and hence  Theorem~\ref{special dirichlet}, using an alternative approach: relating intrinsic approximation on $X$ as above to the set-up of approximation of limit points of a lattice in $\SO(n+1,1)$ by its parabolic fixed points, see \cite{Pat} for the case $n=1$ and  \cite{SV} for the general case. Then, using Theorem \ref{true dirichlet} and \equ{count} one can show that }

More generally, when the rank of $Q$ as above is bigger than one, the uniform versions of  Theorem \ref{true dirichlet} and \ref{special dirichlet} do not always hold. However it is possible to use other methods, in particular a generalization of the dynamical approach developed in the present paper, to prove a non-uniform version of Theorem  \ref{special dirichlet}, with $C$ depending on the point being approximated, and to establish analogues of 
other results from this paper
 for arbitrary rational quadratic varieties. Details will appear in the forthcoming paper \cite{FKMS}.

\end{document}